\documentclass{amsart}
              [1997/02/02 v1.2e PROC Author Class]

\copyrightinfo{2006}
  {American Mathematical Society}

\newtheorem{theorem}{Theorem}[section]

\newtheorem{proposition}[theorem]{Proposition}
\newtheorem{example}[theorem]{Example}

\newtheorem{remark}[theorem]{Remark}

\begin{document}

\title[Separately continuous functions and its dependence on $\aleph$ coordinates]{Separately continuous functions on products and its dependence on $\aleph$ coordinates}

\author{V.V.Mykhaylyuk}
\address{Department of Mathematics\\
Chernivtsi National University\\ str. Kotsjubyn'skogo 2,
Chernivtsi, 58012 Ukraine}
\email{vmykhaylyuk@ukr.net}



\commby{Ronald A. Fintushel}



\begin{abstract}
It is investigated necessary and sufficient conditions on topological spaces $X=\prod\limits _{s\in S}X_s$ and $Y=\prod\limits _{t\in T}Y_t$ for the dependence of every separately continuous functions $f:X\times Y\to \mathbb R$ on at most $\aleph$ coordinates with respect to the first or the second variable.
\end{abstract}

\maketitle
\section{Introduction}

Let $X=\prod\limits _{s\in S}X_s$ be the product of a family of sets $X_s$, $Z$ be a set and $T\subseteq S$. We say that {\it a mapping $f:X\to Z$ concentrated on $T$}, if $f(x')=f(x'')$ for $x',x''\in X$ with $x'_{|_T}=x''_{|_T}$. Moreover, if $|T|\le\aleph$, then we say that  {\it $f$ depends on at most $\aleph$ coordinates}. Let $Y$ be a set. We say that {\it a mapping $g:X\times Y \to Z$ concentrated on $T$ with respect to the first variable}, if the mapping $\varphi :X\to Z^Y$, $\varphi (x)(y)=g(x,y)$, concentrated on $T$. Moreover, if $|T|\le\aleph$, then we say that {\it $g$ depends at most on $\aleph$ coordinates with respect to the first variable}. The notion of dependence of $g:X\times Y \to Z$ at most on $\aleph$ coordinates with respect to the second variable can be introduced analogously. Further $\aleph$ means an infinite cardinal. For an abridgment we shall use the term "depends on $\aleph$ coordinates" instead the term "depends on at most $\aleph$ coordinates".

Let $X$ be a topological space and $\aleph$ be an infinite cardinal. We say that a family $\alpha=(A_i:i\in I)$ of sets $A_i\subseteq X$ is {\it locally finite}, if for every $x\in X$ there exists a neighborhood $U$ of $x$ such that the set $\{i\in I:A_i\cap U\ne \O \}$ is finite, {\it poinwise finite} ($\aleph$-{\it pointwise}), if for every $x\in X$ the set $\{i\in I: x\in A_i\}$ is finite (has the cardinality $\leq\aleph$). Moreover, for arbitrary family $\alpha=(A_i:i\in I)$ the cardinality of $I$ we shall call by the cardinality of the family $\alpha$.

The following properties of a topological space $X$ will be useful for our investigation:

(I$_\aleph$) every locally finite family of open nonempty subsets of $X$ has the cardinality $\leq\aleph$;

(II$_\aleph$) every poinwise finite family of open nonempty subsets of $X$ has the cardinality $\leq\aleph$;

(III$_\aleph$) every $\aleph$-pointwise family of open nonempty subsets of $X$ has the cardinality $\leq\aleph$.

Note that for a topological space the condition (I$_\aleph$) coincides with {\it pseudo-$\aleph^+$-compactness}, where $\aleph^+$ is the next after $\aleph$ cardinal. It was developed in [1] a technic of investigation of properties of topological products which based on Shanin's Lemma [2, p.185]. Using this technic it was proved in [3] (Proposition 1) that topological product has (I$_\aleph$), (II$_\aleph$) or (III$_\aleph$) if each its finite subproduct has the same property. Clearly that (III$_\aleph$) $\Longrightarrow$ (II$_\aleph$) $\Longrightarrow$ (I$_\aleph$).

A topological space $X$ is called {\it $\aleph$-compact} if every open covering ${\mathcal U}$ of $X$ with $|{\mathcal U}|\leq \aleph$ has some finite subcover. In particular, $\aleph_0$-compact space is called {\it countably compact space.}

The dependence of continuous functions on products on $\aleph$ coordinates was investigated in papers of many mathematicians $20$-th century (see [2, p.187]). N.~Noble and M.~Ulmer in [1] obtained a most general result in this direction. They established a close relation between the dependence and pseudo-$\aleph$-compactness. In particular, the following proposition follows from theirs results.

\begin{theorem}\label{th:1.1}[Noble, Ulmer]
Let $X=\prod\limits _{s\in S}X_s$ be a product of a family of nontrivial ($|X_s|>1$) completely regular spaces $X_s$ besides $|S|>\aleph_i$ where $\aleph_i$ is an infinite cardinal. Then every continuous function $f:X\to \mathbb R$ depends on $\aleph_i$ coordinates if and only if $X$ is pseudo-$\aleph_{i+1}$-compact where $\aleph_{i+1}$ is the next after $\aleph_i$ cardinal.
\end{theorem}

It is natural to study conditions of dependence for mappings which satisfy conditions weaker than continuity. In particular, it is actual for separately continuous mappings that is mappings of many variables continuous with respect to each variable. On the other hand, it was detected in [4] that the dependence on countable coordinates of separately continuous functions of two variables closely connected with the investigation of discontinuity points set of separately continuous functions of two variables each of which is the product of metrizable compacts.

Conditions of the dependence on $\aleph$ coordinates of separately continuous functions were studied in [3]. It was obtained in [3] that conditions of dependence are closely connected with the conditions (²)$_\aleph$, (²²)$_\aleph$ and (²²²)$_\aleph$ similarly to continuous functions. It was shown in [3] (Theorem 2) that if every separately continuous function $f:X\times Y \to \mathbb R$, where $X$ is completely regular space, $Y=\prod\limits _{t\in T}Y_t$ is a product of nontrivial completely regular spaces $Y_t$ and $|T|>\aleph$, depends on $\aleph$ coordinates with respect to second variable, then $X\times Y$ has (I$_\aleph$) and at least one of $X$ and $Y$ has (II$_\aleph$). Conversely (Theorems 3,4,5), every separately continuous function $f:X\times Y \to \mathbb R$, where $X$ and $Y$ are the same as in Theorem 2 but not necessarily completely regular, depends on $\aleph$ coordinates with respect to second variable if at least one of the following conditions holds:

$(i)$\quad $X$ has (III$_\aleph$) and $Y$ has (I$_\aleph$);

$(ii)$\quad $X$ has (II$_\aleph$) and $Y$ is a pseudo-${\aleph}_0$-compact;

$(iii)$\quad $X$ is countably compact space and $Y$ has (II$_\aleph$).

In the case of countable compact completely regular spaces $X_s$ and $Y_t$ every separately continuous function $f:X\times Y \to \mathbb R$ on the product of spaces $X=\prod\limits_{s\in S}X_s$ and $Y=\prod\limits _{t\in T}Y_t$, where $|T|>\aleph$, depends on $\aleph$ coordinates with respect to the second variable (or depends on $\aleph$ coordinates as a function on the product $\prod\limits _{s\in S}X_s\times \prod\limits _{t\in T}Y_t$) if and only if $X$ or $Y$ has (II$_\aleph$) (Theorem 6).

Among these results attention is drawn to the following fact. We impose symmetrical conditions on $X$ and $Y$ although we study the dependence with respect to the second variable. This follows from the symmetry of the necessary condition of the dependence (Theorem 2 from [3]) and from the symmetry of the sufficient conditions $(ii)$ and $(iii)$ for countably compact space. In this connection the following questions naturally arise: is it possible to interchange the conditions on $X$ and $Y$ in $(i)$; is it possible to replace the pseudo-$\aleph$-compactness in $(ii)$ and countable compactness in $(iii)$ to (I$_\aleph$); is it true the inverse proposition to Theorem 2 from [3]?

In this paper we show that all these questions have the negative answers. Moreover, we ostend that the Baire property play an important role in the investigation of dependence on $\aleph$ coordinates of separately continuous functions defined on products. In particular, we show that for a Baire space $X$ the first question has the positive answer.

\section{Sufficient conditions of dependence}

A set $S_o\subseteq S$ is called {\it a smallest set on which a mapping $f:X\to Z$ concentrated}, where $X=\prod\limits _{s\in S}X_s$, if $f$ concentrated on $S_0$ and $S_0\subseteq S_1$ for every set $S_1\subseteq S$ on which $f$ concentrated. The notion of {\it smallest set on which a mapping concentrated with respect to some variable} can be introduce analogously. The following proposition (see [4, Corollary]) describe the smallest sets for separately continuous mappings.

\begin{proposition}\label{p:2.1} Let $X=\prod\limits _{s\in S} X_s$ be the topological product of a family of topological spaces $X_s$, $Y$ be a set, $Z$ be a Hausdorff topological space and $f:X\times Y\to Z $ be a mapping continuous with respect to the first variable. Then the set
$$
S_0=\{s\in S:(\exists y\in Y)(\exists u,v\in X)(u{|_{S\setminus \{s\}}}=v{|_{S\setminus \{s\}}}\,\,\,\mbox{and}\,\,\, f(u,y)\ne f(v,y))\}
$$
is a smallest set on which $f$ concentrated with respect to the first variable.
\end{proposition}

The following result give us the possibility to replace a space $Y$ with (III$_\aleph$) to the Cantor's cube.

\begin{theorem}\label{th:2.2} Let $X$ be a topological space, $Y=\prod\limits _{t\in T}Y_t$ be the product of a family of topological spaces $Y_t$ with (III$_\aleph$) and $f:X\times Y\to \mathbb R$ be a separately continuous function which does not depend on $\aleph$ coordinates with respect to the second variable. Then there exist a set $S\subseteq T$ ç $|S|>\aleph$ and separately continuous function $g:X\times Y_1\to \mathbb R$, where $Y_1=\{0,1\}^S$, which does not depend on $\aleph$ coordinates with respect to the second variable.
\end{theorem}

\begin{proof} According to Proposition \ref{p:2.1}, the set
$$
T_0=\{t\in T:(\exists x\in X)(\exists y,z\in Y)(y{|_{T\setminus \{t\}}}=z{|_{T\setminus \{t\}}}\,\,\,\mbox{and}\,\,\, f(x,y)\ne f(x,z))\}
$$
is a smallest set on which $f$ concentrated with respect to the second variable. It follows from the conditions of Theorem that $|T_0|>\aleph$. For every $t\in T_0$ we choose points $x_t\in X$ and $y_t, z_t\in Y$, which featured in the definition of $T_0$. Using the continuity of $f$ with respect to the second variable we take basic open neighborhoods $\tilde{V_t}=\prod\limits _{s\in T}\tilde{V}_s^{(t)}$ and $\tilde{W_t}=\prod\limits _{s\in T}\tilde{W}_s^{(t)}$ of $y_t$ and $z_t$ in $Y$ such that $f(x_t,y)\ne f(x_t,z)$ for every $y\in \tilde{V_t}$ and $z\in \tilde{W_t}$. Put $V_s^{(t)}=\tilde{V}_s^{(t)}\bigcap \tilde{W}_s^{(t)}$ for every $s\in T\setminus \{t\}$, $V_t^{(t)}=\tilde{V}_t^{(t)}$ and $V_t=\prod\limits_{s\in T}V_s^{(t)}$. Note that $V_t$ is an open neighborhood of $y_t$, $V_t\subseteq \tilde{V_t}$. Moreover, for every $y\in V_t$ the point $z\in Y$ which defined by
$$
 z(s) = \left \{\begin{array}{rr}
 y(s),
&
 {\rm if}\quad s\in T\setminus\{t\};
\\
  z_t(t),
&
 {\rm if}\quad  s=t,
  \end{array} \right .
$$
belongs to $\tilde{W_t}$. Thus,  $f(x_t,y)\ne f(x_t,z)$.

We consider the family ${\mathcal V} = (V_t:t\in T_0)$. Since $Y$ has (III$_\aleph$) and $|T_0|>\aleph$, the family ${\mathcal V}$ is not $\aleph$-pointwise. Therefore there exist a set $S\subseteq T_0$ and point $y^{*}\in Y$ such that $|S|>\aleph$ and $y^*\in V_s$ for all $s\in S$. For every $s\in S$ we put $a_s=y^*(s)$, $b_s=z_s(s)$ and
$$
 z^*_s(t) = \left \{\begin{array}{rr}
 y^*(t),
&
 {\rm if}\quad t\in T\setminus\{s\};
\\
  b_t,
&
 {\rm if}\quad  s=t.
  \end{array} \right .
$$
Note that ${y^*}{|_{T\setminus \{s\}}} = {z^*_s}{|_{T\setminus \{s\}}}$ and $f(x_s,y^*)\ne f(x_s,z^*_s)$. Moreover, the space $Y^*=\prod\limits _{t\in S}\{a_t,b_t\}\times \prod\limits _{t\in T\setminus S} \{y^*(t)\}$ is homeomorphic to the space $\{0,1\}^S$, besides the restriction $f^*=f_{|_{X\times Y^*}}$ is a separately continuous mapping, for which $S$ is a smallest set on which $f^*$ concentrated with respect to the second variable. It remains to take the mapping $\varphi :\{0,1\}^S\to Y^*$, which defined by the formula
$$
 \varphi(a)(t) = \left \{\begin{array}{rr}
 a_t,
&
 {\rm if}\quad t\in S \quad {and} \quad a(t)=0;
\\
  b_t,
&
 {\rm if}\quad t\in S \quad {and} \quad a(t)=1;
\\
  y^*(t),
&
 {\rm if}\quad  t\in T\setminus S,
  \end{array} \right .
$$

\noindent and to consider the function $g:X\times Y_1\to \mathbb R$, where $Y_1=\{0,1\}^S$ and $g(x,a)=f^*(x,\varphi(a))$.
\end{proof}

For a basic open set $V=\prod\limits _{t\in T}V_t$ in a topological product $Y=\prod\limits _{t\in T}Y_t$ and a set $B\subseteq T$ the set
$\prod\limits _{t\in B}V_t$ we denote by $V{|_{B}}$ and the set $\{t\in T:V_t\ne Y_t\}$ we denote by $R(V)$.

We shall use the following result which is a reformulation of Shanin's Lemma [2, p. 185] for isolated cardinals.

\begin{proposition}\label{p:2.3} Let $(A_i: i\in I)$ be a family of finite sets $A_i$ with $|I|>\aleph$. Then there exist a finite set $B\subseteq \bigcup\limits_{i\in I}A_i$ and set $J\subseteq I$ such that $|J|>\aleph$ and $A_i\bigcap A_j=B$ for arbitrary distinct $i,j\in J$.
\end{proposition}

Now we consider functions on the product Baire space and Cantor's cube.

\begin{theorem}\label{th:2.4}
Let $X$ be a Baire space with (I$_\aleph$) and $Y=\{0,1\}^T$. Then every separately continuous function $f:X\times Y\to \mathbb R$
depends on $\aleph$ coordinates with respect to the second variable.
\end{theorem}

\begin{proof} Suppose the contrary, that is there exists a separately continuous function $f:X\times Y\to \mathbb R$ which does not depends on $\aleph$ coordinates with respect to $y$. Clearly that $|T|>\aleph$. Then according to Proposition \ref{p:2.1} we have $|T_0|>\aleph$ where
$$
T_0=\{t\in T:(\exists x\in X)(\exists y,z\in Y)(y{|_{T\setminus \{t\}}}=z{|_{T\setminus \{t\}}}\,\,\,\mbox{and}\,\,\, f(x,y)\ne f(x,z))\}.
$$
\noindent For every $n\in {\mathbb N}$ we put
$$
T_n=\{t\in T:(\exists x\in X)(\exists y,z\in Y)(y{|_{T\setminus \{t\}}}=z{|_{T\setminus \{t\}}}\,\,\,\mbox{and}\,\,\,|f(x,y) - f(x,z)|>\frac{1}{n})\}.
$$
\noindent Clearly that $T_0=\bigcup\limits^{\infty}_{n=1} T_n$. Since $\aleph$ is infinite cardinal, there exists integer $n_0$ such that $|T_{n_0}|>\aleph$.

We denote $S=T_{n_0}$, $\varepsilon=\frac{1}{n}$ and for every $t\in S$ choose points $\tilde{x_t}\in X$ and $y_t, z_t \in Y$, which featured in the definition of $T_{n_0}$. Using the continuity of $f$ with respect to $x$ for every $s\in S$ we choose an open neighborhood $\tilde{U_s}$ of $\tilde{x_s}$ in $X$ such that  $|f(x,y_t) - f(x,z_t)|>\varepsilon$ for every $x\in \tilde{U_s}$. Note that the Cantor's cube $Y$ is a co-Namioka space (see [5, Corollary 1.2 or Theorem 2.2]), that is for every Baire space $X'$ and every separately continuous function $g:X'\times Y\to \mathbb R$ there exists a dense $G_{\delta}$-set $A$ in $X'$ such that $g$ is jointly continuous at every point of set $A\times Y$. Therefore in every open nonempty set $\tilde{U_s}$ there exists a point $x_s$ such that $f$ is jointly continuous at points $(x_s,y_s)$ and $(x_s,z_s)$. For every $s\in S$ we take an open neighborhood $U_s$ of $x_s$ in $X$ and basic open neighborhoods $V_s$ and $W_s$ of $y_s$ and $z_s$ in $Y$ such that $R(V_s)=R(W_s)$, ${V_s}{|_{T\setminus \{s\}}} = {W_s}{|_{T\setminus \{s\}}}$ and $|f(x,y)-f(x,z)|>\varepsilon$ for every $x\in U_s$, $y\in V_s$ and $z\in W_s$.

We consider the family $(R(V_s): s\in S)$ which consists of finite sets $R(V_s)$. According to Proposition \ref{p:2.3} there exist a finite set $B\subseteq S$ and a set $S_1\subseteq S$ such that $|S_1|>\aleph$, $B\bigcap S_1=\O$ and $R(V_s)\bigcap R(V_t)=B$ for arbitrary distinct $s, t\in S_1$. Since set $\{0,1\}^B$ is finite and all sets ${V_s}{|_B}$ are nonempty, where $s\in S_1$, there erxists $y'\in \{0,1\}^B$ such that the set $S_0=\{s\in S_1:y'\in {V_s}{|_B}\}$ has the cardinality $>\aleph$.

We show that the family $(U_s:s\in S_0)$ is locally finite in $X$. Let $x_0\in X$. Using the compactness of $Y$ and continuity of the function $f^{x_0}:Y\to \mathbb R$, $f^{x_0}(y)=f(x_0,y)$, we choose a finite set $B_0\subseteq T$ such that $|f(x_0,y)-f(x_0,z)|<\varepsilon$ for every $y,z\in Y$ with $y{|_{B_0}} = z{|_{B_0}}$.
Consider the points $y_0,z_0\in Y$ which defined by
$$
 y_0(t) = \left \{\begin{array}{rr}
 y'(t),
&
 {\rm if}\quad t\in B;
\\
  y_s(t),
&
 {\rm if}\quad t\in R(V_s)\setminus B,\,\, s\in S_0\setminus B_0;
\\
  0,
&
 {\rm otherwise},
  \end{array} \right .
$$
$$
 z_0(t) = \left \{\begin{array}{rr}
 y'(t),
&
 {\rm if}\quad t\in B;
\\
  z_s(t),
&
 {\rm if}\quad t\in R(V_s)\setminus B,\,\, s\in S_0\setminus B_0;
\\
  0,
&
 {\rm otherwise}.
  \end{array} \right .
$$
\noindent Note that according to the choice of sets $B$ and $S_1$, the sets $R(V_s)\setminus B$ are pairwise disjoint for $s\in S_1$. Therefore the definitions of $y_0$ and $z_0$ are correct. Since ${y_s}{|_{S\setminus \{s\}}} = {z_s}{|_{S\setminus \{s\}}}$ for every $s\in S_0$, $y_s(t)=z_s(t)$ for every $t\in B_0$ and $s\in S_0\setminus B_0$. Therefore ${y_0}{|_{B_0}} = {z_0}{|_{B_0}}$ and $|f(x_0,y_0) - f(x_0,z_0)| < \varepsilon$. Using the continuity of $f$ with respect to $x$ we choose a neighborhood $U_0$ of $x_0$ in $X$ such that $|f(x,y_0) - f(x,z_0)| < \varepsilon$ for every $x\in U_0$.

On other hand, since $S_0\bigcap B =\O$, ${V_s}{|_{B}} = {W_s}{|_{B}}$. Therefore $y'\in {W_s}{|_{B}}$ for every $s\in S_0$. Moreover, ${y_0}{|_{R(V_s)\setminus B}} = {y_s}{|_{R(V_s)\setminus B}}$ and ${z_0}{|_{R(V_s)\setminus B}} = {z_s}{|_{R(V_s)\setminus B}}$ for all $s\in S_0\setminus B_0$. Taking into account that $y_s\in V_s$ and $z_s\in W_s$ we obtain that ${y_0}{|_{R(V_s)}} \in {V_s}{|_{R(V_s)}}$ and ${z_0}{|_{R(V_s)}} \in {W_s}{|_{R(V_s)}}$. Therefore $y_0\in V_s$ and $z_0\in W_s$, because $R(W_s)=R(V_s)$ for every $s\in S_0\setminus B_0$. According to the choice of $U_s$, $V_s$ and $W_s$, we have $|f(x,y_0)-f(x,z_0)|>\varepsilon$ for every $x\in U_s$ and $s\in S_0\setminus B_0$.

Thus, $U\bigcap U_s=\O$ for every $s\in S_0\setminus B_0$. Therefore the set $\{s\in S_0:U_0\bigcap U_s\ne \O\}$ is finite and the family $(U_s:s\in S_0)$ is locally finite in $X$. But this contradicts to the fact that $X$ has (I$_\aleph$), because $|S_0|>\aleph$.
\end{proof}

Now Theorems \ref{th:2.2} and \ref{th:2.4} imply the main result of this section.

\begin{theorem}\label{th:2.5} Let $X$ be a Baire space with (I$_\aleph$), $Y=\prod\limits _{t\in T}Y_t$ be the product of a family of topological spaces $Y_t$ which has (I$_\aleph$). Then every separately continuous function $f:X\times Y\to \mathbb R$ depends on $\aleph$ coordinates with respect to the second variable.
\end{theorem}

\section{Essentiality of Baireness}

In this section we show that the condition of Baireness of $X$ in Theorem \ref{th:2.5} is essential.

\begin{proposition}\label{p:3.1}
Let $X$ be a topological space, ${\mathcal U}$ be an infinite locally finite family of open nonempty sets in $X$. Then there exists a family ${\mathcal V}$ of nonempty pairwise disjoint open in $X$ sets such that $|{\mathcal V}|=|{\mathcal U}|$.
\end{proposition}

\begin{proof} Let ${\mathcal U} = (U_{\alpha}:\alpha < \beta)$, where $\beta$ is the first ordinal with the cardinality $|{\mathcal U}|$. We construct the family ${\mathcal V}$ inductively.

Take any point $x_1\in U_1$ and choose an open neighborhood $V_1$ of $x_1$ such that the set $A_1=\{\alpha <\beta : V_1\bigcap U_{\alpha}\ne \O\}$ is finite.

Assume that pairwise disjoint sets $V_{\xi}$ for $\xi <\alpha < \beta$ are constructed such that all sets $A_{\xi}=\{\gamma <\beta : V_{\xi}\bigcap U_{\gamma}\ne \O\}$ are finite. Note that $|\bigcup\limits_{\xi<\alpha} A_{\xi}|< |\beta|$. Really, if $\alpha$ is a finite ordinal, then $|\bigcup\limits_{\xi<\alpha} A_{\xi}|< {\aleph}_0 \leq |\beta|$. If $\alpha$ is an infinite ordinal then $|\bigcup\limits_{\xi<\alpha} A_{\xi}|\leq {\aleph}_0 \cdot |\alpha| = |\alpha| <
|\beta|$. Therefore there exists $\gamma <\beta$ such that $\gamma \not \in \bigcup\limits_{\xi<\alpha} A_{\xi}$, that is $U_{\gamma}\bigcap (\bigcup\limits_{\xi<\alpha} V_{\xi}) = \O$. Take a point $x_{\gamma}\in U_{\gamma}$ and choose an open neighborhood $V_{\alpha}$ of $x_{\gamma}$ such that the set $A_{\alpha}=\{\xi <\beta : V_{\alpha}\bigcap U_{\xi}\ne \O\}$ is finite. The family ${\mathcal V}=(V_{\alpha}:\alpha < \beta)$ is the required.
\end{proof}

Note that in the case of regular space $X$ we can construct a discrete family ${\mathcal V}$ of corresponding cardinality.

\begin{proposition}\label{p:3.2}
Let a topological space $X$ has (II$_{\aleph}$), $Y$ be a dense in $X$ set. Then the subspace $Y$ of $X$ has (I$_{\aleph}$).
\end{proposition}

\begin{proof}.
Let ${\mathcal U}$ be a locally finite family of nonempty open in $Y$ sets. According to Proposition \ref{p:3.1}, there exists a family ${\mathcal V} = (V_i:i\in I)$ of pairwise disjoint nonempty open in $Y$  sets $V_i$ such that $|I|=|{\mathcal U}|$. For every $i\in I$ we choose an open set $U_i$ in $X$ such that $U_i\bigcap Y = V_i$. Note that the family $(U_i:i\in I)$ consists of pairwise disjoint sets. Since $X$ has (II$_{\aleph}$), $|I|\leq \aleph$, that is $|{\mathcal U}|\leq \aleph$. Thus, $Y$ has (I$_{\aleph}$).
\end{proof}

\begin{example}\label{ex:3.3}
Let $Y=[0,1]^S$, $X=C_p(Y)$ be the space of continuous functions on $Y$ with the topology of pointwise convergence and $f:X\times Y \to \mathbb R$ be the calculation function, that is $f(x,y)=x(y)$.

Clearly that $f$ is a separately continuous function. Since each finite power $[0,1]^n$ has (III$_{\aleph}$), according to Proposition 1 from [3],  the space $Y$ has (III$_{\aleph}$). Analogously the space $\mathbb R^Y$ has (II$_{\aleph}$). Taking into account the density $X$ in $\mathbb R^Y$ and using Proposition \ref{p:3.2} we obtain that $X$ has (I$_{\aleph_0}$). Fix $s\in S$. Denote by $x_s$ the function from $X$, which defined by $x_s(y)=y(s)$. Choose points $y_s,z_s \in Y$ such that ${y_s}{|_{S\setminus \{s\}}} = {z_s}{|_{S\setminus \{s\}}}$ and $y_s(s)\ne z_s(s)$, that is $f(x_s,y_s)\ne f(x_s,z_s)$. Thus, according to Proposition \ref{p:2.3}, we have $s\in S_0$, where $S_0$ is a smallest set on which $f$ concentrated with respect to $y$. Hence, $S_0=S$.

Taking $S$ such that $|S|>\aleph$ we obtain an example of separately continuous function defined on the product of spaces $X$ and $Y$, where $X$ has (I$_{\aleph}$) and $Y$ has (III$_{\aleph}$), which does not depends on $\aleph$ coordinates with respect to $y$. Thus, the condition of Baireness of $X$ in Theorem \ref{th:2.4} can not be replaced by (I$_{\aleph}$) or even by (I$_{\aleph_0}$).
\end{example}

\section{Construction some spaces}

In this section we consider some spaces. They will be used in a construction of examples which show the essentiality of corresponding properties of $X$ and $Y$ in the sufficient conditions of dependence on $\aleph$ coordinates of separately continuous functions. Moreover, we obtain a relations between these properties.

\begin{proposition}\label{p:4.1}
Every Hausdorff pseudo-$\aleph_0$-compact space is Baire.
\end{proposition}

\begin{proof} Let $X$ be a pseudo-$\aleph_0$-compact space, that is every locally finite family of open nonempty sets in $X$ is finite. We show that $X$ is Baire.

Let $U_0$ be an open nonempty set in $X$. Suppose that $U_0$ is a meager set. Then there exists a sequence $(F_n)^{\infty}_{n=1}$ of nowhere dense closed sets $F_n\subseteq X$ such that $U_0\subseteq \bigcup\limits^{\infty}_{n=1} F_n$. It easy to construct a decreasing sequence $(U_n)^{\infty}_{n=1}$ of open in $X$ nonempty sets $U_n$ such that $\overline{U_n}\subseteq U_{n-1}$ and $U_n\bigcap F_n = \O$ for every $n\in {\mathbb N}$. Note that
$$
\mathop{\cap}\limits^{\infty}_{n=1} \overline{U_n} = \mathop{\cap}\limits^{\infty}_{n=1} U_n =
(\mathop{\cap}\limits^{\infty}_{n=1} U_n)\mathop{\cap} U_0 = (\mathop{\cap}\limits^{\infty}_{n=1} U_n)\mathop{\cap} (\mathop{\cup}\limits^{\infty}_{n=1} F_n) =
$$
$$
=\mathop{\cup}\limits^{\infty}_{m=1}((\mathop{\cap}\limits^{\infty}_{n=1} U_n)\mathop{\cap} F_m)\subseteq \mathop{\cup}\limits^{\infty}_{n=1}(U_n\mathop{\cap} F_n)=\O.
$$

\noindent Therefore the sequence $(\overline{U_n})^{\infty}_{n=1}$ is locally finite. Thus, the sequence $(U_n)^{\infty}_{n=1}$ is locally finite too. But this contradicts to pseudo-$\aleph_0$-compactness of $X$. Hence, $U_0$ is of second category in $X$. Thus, $X$ is a Baire space.
\end{proof}

Recall that for every function $f:X\to \mathbb R$ the set ${\rm supp}\,f = \{x\in
X: f(x)\ne 0\}$ is called {\it the support of function $f$}.

Let $P_{\aleph}=\{x\in [0,1]^{\aleph^+}: |{\rm supp}\,x|\leq \aleph\}$, that is $P_{\aleph}$ is the subspace of Tikhonoff cube with the weight $\aleph^+$, which consists of all functions with the support of the cardinality $<\aleph^+$.

\begin{proposition}\label{p:4.2}

$a$) $P_{\aleph}$ \,\,has (II$_{\aleph_0}$);

$b$) $P_{\aleph}$ \,\, does not have (III$_{\aleph}$);

$c$) $P_{\aleph}$ \,\, has (III$_{\aleph'}$) for every infinite cardinal $\aleph'\ne\aleph$;

$d$) $P_{\aleph}$ \,\, is $\aleph$-compact.
\end{proposition}

\begin{proof} Let $S$ be a set such that $|S|=\aleph^+$ and $P_{\aleph}\subseteq [0,1]^S$.

$a$). Let ${\mathcal U}=(U_i:i\in I)$ be a family of basic open nonempty sets $U_i$ in $P_{\aleph}$ and $|I|>\aleph_0$. For every $i\in I$ we denote by $\tilde{U_i}$ such basic open nonempty set in $[0,1]^S$ such that $U_i=\tilde{U_i}\bigcap P_{\aleph}$. We consider the family $(B_i: i\in I)$, where $B_i=R(\tilde{U_i})$. According to Proposition \ref{p:2.3}, there exist a finite set $B\subseteq S$ and an uncountable set $I_1\subseteq I$ such that $B_i\bigcap B_j=B$ for every distinct $i,j\in I_1$. The space $[0,1]^B$ has (II$_{\aleph_0}$). Therefore the family $(V_i:i\in I_1)$, where $V_i=\tilde{U_i}{|_{B}}$, is not pointwise finite. Thus, there exist $x_0\in [0,1]^B$ and countable set $I_0\subseteq I_1$ such that $x_0\in V_i$ foe every $i\in I_0$. For every $i\in I_0$ choose a point $x_i\in \tilde{U_i}$ and consider the function $x:S\to [0,1]$, which defined by:
$$
 x(s) = \left \{\begin{array}{rr}
 x_0(s),
&
 {\rm if}\quad s\in B;
\\
  x_i(s),
&
 {\rm if}\quad s\in B_i\setminus B,\,\, i\in I_0;
\\
  0,
&
 {\rm if}\quad s\in S\setminus (\bigcup\limits_{i\in I_0}B_i).
  \end{array} \right .
$$
Since $B_i\bigcap B_j=B$, $(B_i\setminus B)\bigcap (B_j\setminus B)=\O$ for arbitrary distinct $i,j\in I_0$. Therefore the function $x$ is defined correctly. Moreover, ${\rm supp}\,x \subseteq \bigcup\limits_{i\in I_0}B_i$ and $|\bigcup\limits_{i\in I_0}B_i|\leq |I_0|\cdot \aleph_0 \leq \aleph$. Thus, $|{\rm supp}\,x|\leq \aleph$. Hence, $x\in P_{\aleph}$. Note that $x{|_B}=x_0\in V_i={\tilde{U}_i}{|_B}$ and $x{|_{B_i\setminus B}} = {x_i}{|_{B_i\setminus B}}\in {\tilde{U_i}}{|_{B_i\setminus
B}}$. Therefore, $x{|_{B_i}}\in {\tilde{U_i}}{|_{B_i}}$, that is $x\in \tilde{U_i}$ for every $i\in I_0$. Thus, $x\in U_i$ for every $i\in I_0$ and the family ${\mathcal U}$ is not pointwise finite. Thus, $P_{\aleph}$ has (II$_{\aleph_0}$).

$b$). For every $s\in S$ put $U_s=\{x\in P_{\aleph}: x(s)>0\}$. Clearly that the family $(U_s:s\in S)$ is $\aleph$-pointwise in $P_{\aleph}$, besides $|S|>\aleph$. Thus, $P_{\aleph}$ does not have (III$_{\aleph}$).

$c$). Let $\aleph'$ be an infinite cardinal which not equals to $\aleph$ and ${\mathcal U}=(U_i:i\in I)$ be a family of nonempty open basic sets in $P_{\aleph}$, besides $|I|={\aleph'}^+$. For every $i\in I$ we choose a basic open set $\tilde{U_i}$ in $[0,1]^S$ such that $U_i=\tilde{U_i}\bigcap P_{\aleph}$ and put $B_i=R(\tilde{U_i})$.

Firstly we consider the case of $\aleph'<\aleph$. Put $B=\bigcup\limits_{i\in I}B_i$. The space $[0,1]^B$ has (III$_{\aleph'}$), therefore the family $(V_i:i\in I)$, where $V_i={\tilde{U_i}}{|_B}$, is not $\aleph'$-pointwise, that is there exist $x\in [0,1]^B$ and $I_1\subseteq I$ such that $|I_1|>\aleph'$ and $x\in V_i$ for every $i\in I_1$. Define a point $y\in [0,1]^S$ by:
$$
 y(s) = \left \{\begin{array}{rr}
 x(s),
&
 {\rm if}\quad s\in B;
\\
   0,
&
 {\rm if}\quad s\in S\setminus B.
  \end{array} \right .
$$
Since ${\rm supp}\,y \subseteq B$ and $|B|\leq |I|\cdot \aleph_0 \leq {\aleph'}^+\cdot \aleph_0\leq \aleph$, $y\in P_{\aleph}$. Now it easy to see that $y\in U_i$ fpr every $i\in I_1$. Thus, ${\mathcal U}$ is not $\aleph'$-pointwise. Hence, $P_{\aleph}$ has (III$_{\aleph'}$).

Now let $\aleph'> \aleph^+$. We consider the system ${\mathcal A}$ of all finite subsets of set $S$. Clearly that $|{\mathcal A}|=|S|=\aleph^+$. For every $A\in {\mathcal A}$ we put $I_A=\{i\in I: B_i=A\}$.

Suppose that $|I_A|\leq \aleph'$ for every $A\in {\mathcal A}$. Then $|I|=|\bigcup\limits_{A\in {\mathcal A}}I_A|\leq |{\mathcal A}|\cdot \aleph'=\aleph^+\cdot
\aleph'\leq \aleph'$. But this contradicts to $|I|={\aleph'}^+$. Thus, there exist $A_0\in {\mathcal A}$ and $I_0\subseteq I$ such that$|I_0|>\aleph'$ and $B_i=A_0$ for every $i\in I_0$.

Taking into account that $[0,1]^{A_0}$ has (III$_{\aleph'}$) we obtain that there exist $x\in [0,1]^{A_0}$ and $I_1\subseteq I_0$ such that $|I_1|>\aleph'$ and $x\in {\tilde{U_i}}_{|_{A_0}}$. Then a point $y\in [0,1]^S$ which defined by:
$$
 y(s) = \left \{\begin{array}{rr}
 x(s),
&
 {\rm if}\quad s\in A_0;
\\
   0,
&
 {\rm if}\quad s\in S\setminus A_0,
  \end{array} \right .
$$
belongs to $P_{\aleph}$, besides $y\in U_i$ for every $i\in I_1$.

Thus,  ${\mathcal U}$ is not $\aleph'$-pointwise and $P_{\aleph}$ has (III$_{\aleph'}$).

$d$). This assertion follows from the next fact: the closure in $P_{\aleph}$ of any set with the cardinality $\leq\aleph$ is compact.
\end{proof}

This proposition implies that all properties (III$_{\aleph}$) are not comparable, although it easy to see that (I$_{\aleph}$) $\Longrightarrow$ (I$_{\aleph'}$) and (II$_{\aleph}$) $\Longrightarrow$ (II$_{\aleph'}$) for $\aleph<\aleph'$. Moreover, the discrete space with the cardinality $\aleph$ shows that the implication (III$_{\aleph}$) $\Longrightarrow$ (I$_{\aleph'}$) is not true for $\aleph>\aleph'$. The Aleksandroff compactification (see [2, p. 261]) of discrete space with the cardinality $\aleph^+$ show that the compactness does not imply any properties (II$_{\aleph}$) and (III$_{\aleph}$).

Thus, the properties which featured in the investigation of dependence on $\aleph$ coordinates of separately continuous functions can be represented as the following table. Note that the implications from Proposition \ref{p:4.1} are true for Hausdorff spaces and all non-marked implications are not true.
$$
  \begin{array}{ccccccccccccc}
  \phantom {n}^{\small{\mbox {countably}}}_{\small{\mbox {compact.}}}\!
&
 \Rightarrow\!
&
 \phantom {m}^{{\small{\mbox{ pseudo}}-\aleph_0-}}_{ \,\,\,\small{\mbox{compact.}}}\!
&
 \Rightarrow\!
&
(I_{\aleph_0})\!
&
\Rightarrow\!
&
(I_{\aleph_1})\!
&
\Rightarrow\!
&
\dots\!
&
\Rightarrow\!
&
(I_{\aleph})\!
&
\Rightarrow\!
&
\dots\!
\\
\Uparrow\!
&
&
\Downarrow\!
&
&
\Uparrow\!
&
&
\Uparrow\!
&
&
&
&
\Uparrow\!
&
&
\\
  \phantom {n}^{\small{\aleph-{\mbox{ compact.}}}}_{\,\,\,\,\,\,\small{\mbox{í³ñòü}}}\!
&
&
 {\mbox{ Baire}}\!
&
&
(II_{\aleph_0})\!
&
\Rightarrow\!
&
(II_{\aleph_1})\!
&
\Rightarrow\!
&
\dots\!
&
\Rightarrow\!
&
(II_{\aleph})\!
&
\Rightarrow\!
&
\dots\!
\\
\Uparrow\!
&
&
&
&
\Uparrow\!
&
&
\Uparrow\!
&
&
&
&
\Uparrow\!
&
&
\\
 \phantom {n}^{\small{\mbox{ compact.}}}_{\,\,\,\small{\mbox{í³ñòü}}}\!
&
&
&
&
(III_{\aleph_0})\!
&
&
(III_{\aleph_1})\!
&
&
\dots\!
&
&
(III_{\aleph})\!
&
&
\dots\!
\end{array}
$$

We will use the following auxiliary result in investigation of properties of the second multiplier.

\begin{proposition}\label{p:4.3}
Let $X$ be a topological space with $(II_{\aleph})$, ${\mathcal U}= (U_i:i\in I)$ be a family of open nonempty sets in $X$ with $|I|>\aleph$. Then there exists $x_0\in X$ such that for every neighborhood $U$ of $x_0$ in $X$ the set $\{i\in I: U_i\bigcap U\ne \O\}$ has the cardinality $>\aleph$.
\end{proposition}

\begin{proof} Suppose the contrary. That is for every $x\in X$ there exists an open neighborhood $V_x$ of $x$ such that the cardinality of the set $\{i\in I: U_i\bigcap V_x\ne \O\}$ is $\leq\aleph$.

Using transfinite induction we construct a family $(G_{\alpha}: \alpha <\beta)$, where $\beta$ is the first ordinal of the cardinality $|I|$, of open nonempty pairwise disjoint sets $G_{\alpha}$. This contradicts to (II$_{\aleph}$) of $X$.

We take $i_1\in I$ and $x_1\in U_{i_1}$ and put $G_1=V_{x_1}\bigcap U_{i_1}$. Clearly that the set $I_1=\{i\in I: U_i\bigcap G_1\ne \O\}$ has the cardinality  is $\leq\aleph$.

Assume that the sets $G_{\alpha}$ for $\alpha < \gamma < \beta$ are constructed, moreover all sets $I_{\alpha}=\{i\in I: U_i\bigcap G_{\alpha}\ne \O\}$ have the cardinality is $\leq\aleph$. Note that $|\bigcup\limits_{\alpha < \gamma} I_{\alpha}|\leq |\gamma|\cdot \aleph \leq \aleph^2 = \aleph < |I|$. Therefore there exists $i_{\gamma}\in I$ such that $U_{i_{\gamma}}\bigcap G_{\alpha}=\O$ for every $\alpha < \gamma$. Take a point $x_{\gamma}\in U_{i_{\gamma}}$ and put $G_{\gamma} = U_{i_{\gamma}}\bigcap V_{x_{\gamma}}$. Clearly that the cardinality of $I_{\gamma}$ is $\leq\aleph$.
\end{proof}

Denote by $D_{\aleph}$ the discrete space of the cardinality $\aleph^+$. Let $Q_{\aleph}=D_{\aleph}\bigcup \{\infty\}$, besides neighborhoods of $\infty$ in
$Q_{\aleph}$ are all sets $A\bigcup\{\infty\}$, where $A\subseteq D_{\aleph}$ and $|D_{\aleph}\setminus A|\leq \aleph$.

\begin{proposition}\label{p:4.4}
$a$)\,\,\, $Q_{\aleph}$ has $(I_{\aleph})$;

$b$)\,\,\, every subfamily of $(\{d\}:d\in D_{\aleph})$ with the cardinality $\aleph$ is locally finite in $Q_{\aleph}$;

$c$)\,\,\, $P_{\aleph}\times Q_{\aleph}$ has $(I_{\aleph})$.
\end{proposition}

\begin{proof} $a$). Let ${\mathcal U} = (U_i: i\in I)$ be a locally finite family of nonempty open in $Q_{\aleph}$ sets $U_i$. We choose a neighborhood $U$ of $\infty$ such that the set $J=\{i\in I: U_i\bigcap U \ne \O\}$ is finite. The set $B=D_{\aleph}\setminus U$ has the cardinality $\leq\aleph$. Moreover, for every $b\in B$ the set $I_b=\{i\in I: b\in U_i\}$ is finite. Since $I=J\bigcup (\bigcup\limits_{b\in B} I_b)$, $|I|\leq \aleph_0 + \aleph\cdot \aleph_0 \leq
\aleph$.

$b$). It follows immediately from the definition of $D_{\aleph}$.

$c$). Let ${\mathcal W}=(W_i:i\in I)$ be a locally finite family of nonempty basic open sets $W_i=U_i\times V_i$ in $R=P_{\aleph}\times Q_{\aleph}$. We may assume that $V_i=\{d_i\}$, where $d_i\in D_{\aleph}$. For every $d\in D_{\aleph}$ we put $I_d=\{i\in I: V_i=\{d\}\}$. Clearly that the sets $I_d$ are pairwise disjoint. Since the family ${\mathcal W}$ is locally finite in $R$, for every $d\in D_{\aleph}$ the family $(U_i: i\in I_d)$ is locally finite in $P_{\aleph}$. Since $P_{\aleph}$ has (II$_{\aleph_0}$), $|I_d|\leq \aleph_0$ for every $d\in D_{\aleph}$.

Now we show that the set $B=\{d\in D_{\aleph}: I_d\ne \O\}$ has the cardinality $\leq\aleph$.

Suppose that $|B|>\aleph$. For every $b\in B$ we choose $i_b\in I_b$ and put $U_b=U_{i_b}$, $V_b=V_{i_b}$. We consider the family $(U_b:b\in B)$. Since $P_{\aleph}$ has (II$_{\aleph}$), according to Proposition \ref{p:4.3} there exists $p\in P_{\aleph}$ such that for every neighborhood $U$ of $p$ the set $B_U=\{b\in B: U\bigcap U_b \ne \O\}$ has the cardinality $>\aleph$. For every neighborhood $V$ of $\infty$ in $Q_{\aleph}$ we have $|D_{\aleph}\setminus V|\leq \aleph$. Therefore, $|B_U\setminus V|\leq \aleph$, thus, the set $B_0=B_U\bigcap V$ has the cardinality $>\aleph$. For every $b\in B_0$ we have $U_b\bigcap U \ne\O$ and $V_b=\{b\}\subseteq V$. Hence, $B_0\subseteq \{b\in B: (U\times V)(U_b\times V_b)\ne \O\}$ and $\{i_b: b\in B_0\}\subseteq \{i\in I: (U\times V)\bigcap W_i\ne \O\}=I'$. Since all indexes $i_b$ are distinct for $b\in B_0$, because $V_{i_b}=V_b=\{b\}$, the set $I'$ is infinite. Thus, the family  ${\mathcal W}$ is not locally finite in the point $(p,\infty)$,  a contradiction.

Thus, $|B|\leq \aleph$. Then, taking into account that $I\subseteq \bigcup\limits_{b\in B}I_b$ and $|I_b|\leq \aleph_0$ we obtain that $|I|\leq \aleph_0\cdot |B|\leq \aleph_0\cdot \aleph = \aleph$.

Thus, $R$ has (I$_{\aleph}$).
\end{proof}

\section{Essentiality of some sufficient conditions}

Firstly we generalize a construction of separately continuous functions from [3].

Families ${\mathcal U}=(U_i: i\in I)$ and ${\mathcal V}=(V_i: i\in I)$ of sets $U_i$ and $V_i$ in topological spaces $X$ and $Y$ respectively are called {\it concerted}, if for every $x\in X$ and $y\in Y$ the families ${\mathcal V}_x=(V_i: i\in I^x)$ and ${\mathcal U}_y=(U_i:i\in I_y)$, where $I^x=\{i\in I: x\in U_i\}$ and $I_y=\{i\in I: y\in V_i\}$, are locally finite in $Y$ and $X$ respectively. Clearly that pointwise finite families ${\mathcal U}$
and ${\mathcal V}$ are concerted.

\begin{theorem}\label{th:5.1}
Let $X$, $Y$ be topological spaces, $({\varphi}_i:i\in I)$, $({\psi}_i:i\in I)$ be families of continuous functions ${\varphi}_i:X\to \mathbb R$ and
${\psi}_i:Y\to \mathbb R$ such that the families ${\mathcal U} = ({\rm supp}\,{\varphi}_i: i\in I)$ and ${\mathcal V} = ({\rm supp}\,{\psi}_i: i\in I)$ are concerted, $(f_i:i\in I)$ be a family of separately continuous functions $f_i:X\times Y\to \mathbb R$. Then the function $f:X\times Y\to \mathbb R$, which defined by formula
$$
f(x,y)=\sum\limits_{i\in I}{\varphi}_i(x){\psi}_i(y)f_i(x,y),
$$
is separately continuous.
\end{theorem}

\begin{proof} Fix $x_0\in X$ and put $I_0=\{i\in I: x_0\in {\rm supp}\,{\varphi}_i\}$. Since the families ${\mathcal U}$ and ${\mathcal V}$ are concerted, the family $({\rm supp}\,{\psi}_i: i\in I_0)$ is locally finite on $Y$. Therefore the locally finite sum $\sum\limits_{i\in I_0}{\varphi}_i(x_0) {\psi}_i(y)f_i(x_0,y)$ of continuous on $Y$ functions ${\varphi}_i(x_0){\psi}_i(y)f_i(x_0,y)$ is continuous on $Y$. On the other hand, $f(x_0,y)=\sum\limits_{i\in I}{\varphi}_i(x_0){\psi}_i(y)f_i(x_0,y)$. Thus, the function $f^{x_0}$ is continuous on $Y$.

The continuity of the function $f$ with respect $y$ can be verify analogously.
\end{proof}

\begin{theorem}\label{th:5.2} Let $Z=[0,1]^S$, where $|S|=\aleph^+$, $X_1=P_{\aleph}$,
$Y_1=Q_{\aleph}\times Z$, $X_2=Q_{\aleph}$, $Y_2=P_{\aleph}\times Z$. Then there exists a function $f:P_{\aleph}\times Q_\aleph\times Z\to \mathbb R$ which separately continuous on $X_1\times Y_1$ and on $X_2\times Y_2$ and such that $S$ is a smallest set on which $f$ concentrated with respect to $z$.
\end{theorem}

\begin{proof} Let $P_\aleph = \{x\in [0,1]^S: |{\rm supp}\,x|\leq \aleph\}$ and $Q_\aleph=S\bigcup \{\infty\}$. We consider the families $(\varphi_s: s\in S)$, $(\psi_s: s\in S)$ and $(f_s: s\in S)$ of continuous functions $\varphi_s: P_\aleph \to \mathbb R$, $\varphi_s(p)=p(s)$, $\psi_s:Q_\aleph \to \mathbb R$, $\psi_s(q)=
\left\{\begin{array}{ccc} 0, \,\,\,{\rm if}\,\,\,q\ne s;\\ 1,\,\,\,{\rm if}\,\,\,
q=s,
\end{array}\right.$ and $f_s:Z\to \mathbb R$, $f_s(z)=z(s)$. It follows from the definition of $P_\aleph$ that the family $({\rm supp}\,\varphi_s: s\in S)$ is $\aleph$-pointwise. According to Proposition \ref{p:4.4}, every subfamily of the family $({\rm supp}\,\psi_s: s\in S)$ with the cardinality $\aleph$ is locally finite. Therefore for every $p\in P_\aleph$ the family $({\rm supp}\,\psi_s: p\in {\rm supp}\,\varphi_s)$ is locally finite in $Q_\aleph$. On other hand, the family $({\rm supp}\,\psi_s: s\in S)$ is pointwise finite on $Q_\aleph$. Thus, the families $({\rm supp}\,\varphi_s: s\in S)$ and $({\rm supp}\,\psi_s: s\in S)$ are concerted.

Acording to Theorem\ref{th:5.1}, the function
$$
f(p,q,z)=\sum\limits_{s\in S}\varphi_s(p)\psi_s(q)f_s(z)
$$
is separately continuous on $X_1\times Y_1$ and $X_2\times Y_2$.

For every $s\in S$ we denote by $p_s$ the characteristic function of the set $\{s\}$ on $S$. Moreover, let $z_0\equiv 0$ on $S$. Then $f(p_s,s,p_s)=1$ and $f(p_s,s,z_0)=0$, besides ${p_s}{|_{S\setminus \{s\}}}={z_0}{|_{S\setminus \{s\}}}$. Thus, according to Proposition \ref{p:2.1}, $S$ is a smallest set on which $f$ concentrated with respect to $z$.
\end{proof}

\begin{remark}
Note that $X_1$ is $\aleph$-compact with (II$_{\aleph_0}$) (Proposition \ref{p:4.2}) and $Y_1$ and $X_1\times Y_1$ have (I$_{\aleph}$). Therefore the function from Theorem \ref{th:5.2} shows that in the sufficient condition $(i)$ the property (III$_{\aleph}$) of $X$ can not be replaced by (II$_{\aleph_0}$) with $\aleph$-compactness; and in the sufficient condition $(ii)$ the pseudo-$\aleph_0$-compactness of $Y$ can not be replaced by (I$_{\aleph}$). On other hand, Baire space $X_2$ has (I$_{\aleph}$) and $Y_2$ is $\aleph$-compact with (II$_{\aleph_0}$). Therefore the countable compactness of $X$ in the sufficient condition $(iii)$ can not be replaced by (I$_{\aleph}$); and the property (III$_{\aleph}$) of $Y$ in Theorem \ref{th:2.5} can not be replaced by (II$_{\aleph_0}$).
\end{remark}

\bibliographystyle{amsplain}

\end{document}